\documentclass[reqno]{amsart}

\pdfoutput=1
\usepackage{amssymb,amsmath,latexsym}
\usepackage{empheq}

\newtheorem{theorem}{Theorem}
\newtheorem{lemma}[theorem]{Lemma}

\newtheorem{definition}[theorem]{Definition}
\newtheorem{proposition}[theorem]{Proposition}
\newtheorem*{theorem*}{Theorem}

\newtheorem{corollary}[theorem]{Corollary}
\newtheorem*{conjecture*}{Conjecture}

\DeclareMathOperator{\sign}{sign}

\title{A vector-contraction inequality for Rademacher complexities using $p$-stable variables}

\author{Oscar Zatarain-Vera}
\address{Department of Mathematical Sciences, Kent State University, Kent, OH 44242, USA}
\email{ozatarai@kent.edu}

\begin{document}
    \maketitle

    \begin{abstract}
        Andreas Maurer in the paper ``A vector-contraction inequality for Rademacher complexities'' \cite{Maurer} extended the contraction inequality for Rademacher averages to Lipschitz functions with vector-valued domains; He did it replacing the Rademacher variables in the bounding expression  by arbitrary idd symmetric and sub-gaussian variables. We will see how to extend this work when we replace sub-gaussian variables by $p$-stable variables for $1<p< 2$.
    \end{abstract}
    \vspace{1cm}

    The main result of this work is based in the following class of random variables:

    \begin{definition}
        A real valued symmetric random variable $X$ is called $p-$stable, $0<p\leq 2$ if for some $\sigma\geq 0$, its Fourier transform (characteristic function) is of the form
        \[
            \mathbb{E}\exp(itX)=\exp(-\sigma^{p}|t|^p/2).
        \]
    \end{definition}

    The $p-$stable random variables are characterized by their fundamental ``stability'' property: if $(X_i)$ is a $p-$stable sequence, i.e. a sequence of independent $p-$stable random variables $X_i$, then for any finite sequence $(\alpha_i)$ of real numbers, $\sum_i \alpha_i X_i$ has the same distribution as
    $\left(\sum_i|\alpha_i|^p\right)^{1/p}X_1$. In particular, for any $r<p$,
    \begin{equation}\label{p-stability}
        \|\sum_i \alpha_i X_i\|_r=c_{p,r}\left(\sum_i|\alpha_i|^p\right)^{1/p}.
    \end{equation}
    To know more about $p-$stable random variables, see \cite{LeTa}.

    In general, the preceding random variables are part of a larger class which are known as stable distributions. Let's recall some definitions and notation about them that can be found in \cite{Nolan}.

    \begin{definition}
        A random variable $X$ is said to be \emph{stable} if and only if $X\stackrel{d}{=}\gamma Z+\delta$ where $0<\alpha\leq 2,\hspace{0.1cm} -1\leq \beta\leq 1,\hspace{0.1cm} \gamma\neq 0,
        \hspace{0.1cm} \delta\in\mathbb{R}$ and $Z=Z(\alpha,\beta)$ is a random variable with characteristic function
        \begin{equation}\label{first}
            \mathbb{E}\exp(iuZ)=\begin{cases}
                \exp\left(-|u|^{\alpha}[1-i\beta\tan\frac{\pi\alpha}{2}(\sign u)]\right) , \hspace{2cm} \alpha\neq 1\\
                \exp\left(-|u|[1+i\beta\frac{2}{\pi}(\sign u)\log|u|]\right), \hspace{2cm} \alpha=1
            \end{cases}
        \end{equation}
    \end{definition}

    Notice we need 4 parameters to describe a stable random variable, but sometimes a fifth one is used since there could be different parameterizations $S(\alpha,\beta,\gamma,\delta;k)$. The parameters are $\alpha$ for the characteristic factor, $\beta$ for the skewness, $\gamma$ for the scale, $\delta$ for the location and $k$ for possible different parameterizations.

    \begin{definition}
        A random variable $X$ is $S(\alpha,\beta,\gamma,\delta;0)$ if
        \[
            X\stackrel{d}{=}\begin{cases}
                \gamma\left(Z-\beta\tan\frac{\pi\alpha}{2}\right)+\delta, \hspace{2cm} \alpha\neq 1\\
                \gamma Z+\delta, \hspace{4cm} \alpha=1
            \end{cases}
        \]
        where $Z=Z(\alpha,\beta)$ is given by (\ref{first}).
    \end{definition}

    Note that $X$ has the characteristic function
    \[
        \mathbb{E}\exp(iuX)=\begin{cases}
                \exp\left(-\gamma^{\alpha}|u|^{\alpha}[1+i\beta(\tan\frac{\pi\alpha}{2})(\sign u)(|\gamma u|^{1-\alpha})]+i\delta u\right), \hspace{2cm} \alpha\neq 1\\
                \exp\left(-\gamma|u|[1+i\beta\frac{2}{\pi}(\sign u)\log(\gamma|u|)]+i\delta u\right), \hspace{3.5cm} \alpha=1
            \end{cases}
    \]

    One important characteristic of stable random variables is the tail approximation behavior.

    \begin{theorem}\label{tailbehavior}
        Let $X\sim S(\alpha,\beta,\gamma,\delta;0)$ with $0<\alpha<2, \hspace{0.15cm} -1<\beta<1$. Then as $x\to\infty$ we have
        \[
            P(X>x)\sim \gamma^{\alpha}c_{\alpha}(1+\beta)x^{-\alpha}.
        \]
    \end{theorem}

    Two other properties of stable random variables are how they behave under scalar multiplication and under addition.

    \begin{theorem}[See \cite{Nolan}]\label{properties}
        The $S(\alpha,\beta,\gamma,\delta;0)$ parametrization has the following properties.
        \begin{enumerate}
            \item if $X\sim S(\alpha,\beta,\gamma,\delta;0)$, then for any $a\neq 0, b\in \mathbb{R}$,
            \[
                aX+b \sim S(\alpha,(\sign a)\beta,|a|\gamma,a\delta +b;0).
            \]
            \item The characteristic functions, densities and distribution functions are jointly continuous in all four parameters $(\alpha,\beta,\gamma,\delta)$ and in $x$.
            \item If $X_1 \sim S(\alpha,\beta_1,\gamma_1,\delta_1;0)$ and $X_2 \sim S(\alpha,\beta_2,\gamma_2,\delta_2;0)$ are independent, then $X_1+X_2 \sim S(\alpha,\beta,\gamma,\delta;0)$ where
                \[
                    \beta=\frac{\beta_1\gamma_1^{\alpha}+\beta_2\gamma_2^{\alpha}}{\gamma_1^{\alpha}+\gamma_2^{\alpha}}, \hspace{1cm} \gamma^{\alpha}=\gamma_1^{\alpha}+\gamma_2^{\alpha},
                \]
                \[
                \delta=\begin{cases}
                \delta_1+\delta_2+(\tan\frac{\pi\alpha}{2})[\beta\gamma-\beta_1\gamma_1-\beta_2\gamma_2] , \hspace{3.5cm} \alpha\neq 1\\
                \delta_1+\delta_2+\frac{2}{\pi}[\beta\gamma\log\gamma-\beta_1\gamma_1\log\gamma_1-\beta_2\gamma_2\log\gamma_2], \hspace{2cm} \alpha=1
            \end{cases}
            \]
        \end{enumerate}
    \end{theorem}

    The formula $\gamma^{\alpha}=\gamma_1^{\alpha}+\gamma_2^{\alpha}$ in the third part is the generalization of the rule for adding variances of independent random variables: $\sigma^2=\sigma_1^2+\sigma_2^2$. Note that one adds the $\alpha^{th}$ power of the scale parameters, not the scale parameters themselves.

    We will focus in a particular case of a stable random variable $S(\alpha,\beta,\gamma,\delta;0)$.
    Consider the stable random variable $X=S(p,0,\frac{\sigma}{\sqrt[p]{2}},0;0)$. Since its characteristic function is
    \[
        \mathbb{E}\exp(itX)=\exp\left(-(\frac{\sigma}{\sqrt[p]{2}})^{p}|t|^p\right)=\exp(-\sigma^{p}|t|^p/2).
    \]
    then we will refer to $X$ as a $p-$stable random variable. Recall that the sum of $p-$stable random variables is again $p-$stable. We will emphasize the parameters of the stable random variable $S$ in question whenever it's needed.
    Also, by Theorem \ref{tailbehavior} as $t\to\infty$ we have $P(X>t)\sim \sigma^p c_p t^{-p}$.\\

    In Machine Learning the so-called contraction inequality is widely used. A function $h:\mathbb{R}\to\mathbb{R}$
    is called Lipschitz with constant $L>0$ if for all $x$ and $y$ we have $|h(x)-h(y)|\leq L|x-y|$. For Lipschitz functions $h_i:\mathbb{R}\to\mathbb{R}$ with constant $L$, the scalar contraction inequality (\cite{LeTa, Mohri}) states that
    \begin{equation}\label{contraction}
        \mathbb{E}\sup_{f\in\mathcal{F}}\sum_{i=1}^{n}\epsilon_i h_i(f(x_i))\leq L\mathbb{E}\sup_{f\in\mathcal{F}}\sum_{i=1}^{n}\epsilon_i f(x_i).
    \end{equation}

    Andreas Maurer in \cite{Maurer} extended the contraction inequality to Lipschitz functions with vector domains. Furthermore he also showed that in the bounding expression the Rademacher variables can be replaced by arbitrary iid symmetric and subgaussian variables. Specifically he proved:
     \begin{theorem}[Maurer A., Vector contraction-inequality for subgaussian variables]
        Let $X$ be a nontrivial, symmetric and subgaussian random variable. Then there exists a constant $C<\infty$, depending only on the distribution of $X$, such that for any set $S$ and functions $\psi_i:S\to \mathbb{R}, \phi_i:S\to\ell_2$, $1\leq i\leq n$ satisfying
        \[
            \forall s,s'\in S, \psi_i(s)-\psi_i(s')\leq \|\phi_i(s)-\phi_i(s')\|_2
        \]
        we have
        \[
            \mathbb{E}\sup_{s\in S}\sum_i\epsilon_i\psi_i(s)\leq C\mathbb{E}\sup_{s\in S}\sum_{i,k}X_{ik}\phi_i(s)_k
        \]
        where the $X_{ik}$ are independent copies of $X$ for $1\leq i\leq n$ and $1\leq k\leq \infty$ and $\phi_i(s)_k$ is the $k-th$ coordinate of $\phi_i(s)$.
    \end{theorem}

    Let's see why the preceding is an extension of the contraction inequality. For applications in learning theory, we can do at once:
    \begin{itemize}
        \item Replace a Rademacher variable for $X$ and $\sqrt{2}$ for $C$.
        \item For $S$ we take a class $\mathcal{F}$ of vector valued functions $f:\mathcal{X} \to \ell_2 $.
        \item For the $\phi_i$ the evaluation functionals on a sample $(x_1,...,x_n)$, such that $\phi_i(f)=f(x_i)$.
        \item For $\psi_i$ we take the evaluation functionals composed with a Lipschitz loss function $h:\ell_2\to\mathbb{R}$ of Lipschitz norm $L$.
    \end{itemize}

    In such a way, we get
    \begin{corollary}
        Let $\mathcal{X}$ be any set, $(x_1,...,x_n)\in\mathcal{X}^n$, let $\mathcal{F}$ be a class of functions $f:\mathcal{X}\to\ell_2$ and let $h_i:\ell_2\to\mathbb{R}$ have Lipschitz norm $L$. Then
        \[
            \mathbb{E} \sup_{f\in \mathcal{F}}\sum_i \epsilon_i h_i(f(x_i)) \leq \sqrt{2}L \mathbb{E} \sup_{f\in\mathcal{F}} \sum_{i,k} \epsilon_{ik}f_k(x_i),
        \]
        where $\epsilon_{ik}$ is an independent doubly indexed Rademacher sequence and $f_k(x_i)$ is the $k-$th component of $f(x_i)$.
    \end{corollary}

    We adapt Maurer's proof to extend the vector contraction inequality to $p-$stable random variables where $1< p<2$. To do that we go through a series of results.

    \begin{proposition}\label{prop}
        Let $X$ be a nontrivial $p-$stable random variable, $1<p<2$. Let $\textbf{X}=(X_1,...,X_k,...)$ be a sequence of independent copies of $X$. Then
        \begin{enumerate}
            \item[(i)] For every $v\in \ell_p$, the sequence of random variables $Y_k=\sum_{k=1}^{K}v_kX_k$ converges in $L_r$ for $0<r<p$ to a random variable denoted by $\sum_{k=1}^{\infty}v_kX_k$. So the map $v\mapsto \sum_{k=1}^{\infty}v_kX_k$ is a bounded map from $\ell_p$ to $L_r$.
            \item[(ii)] There exists a constant $C(p)<\infty$ such that for every $v\in\ell_p$
                \[
                    \|v\|_{p}\leq C(p) \mathbb{E}\left|\sum_{k=1}^{\infty}v_kX_k\right|.
                \]
        \end{enumerate}
    \end{proposition}
    \begin{proof}
        Let $X=S(p,0,\frac{\sigma}{\sqrt[p]{2}},0;0)$ be a $p-$stable random variable where $1<p<2$ and $0<r<p$.
        For any $v\in\ell_p$ note that $\displaystyle\sum_{k=1}^{K}v_kX_k$ is again a $p-$stable distribution $S\left(p,0,\displaystyle\frac{1}{\sqrt[p]{2}}\sigma\left(\displaystyle\sum_{k=1}^{K}|v_k|^p\right)^{1/p},0;0\right) =
        S\left(p,0,\displaystyle\frac{1}{\sqrt[p]{2}}\sigma\|v\|_p,0;0\right)$ due to Theorem \ref{properties}.
        \begin{enumerate}
            \item[(i)] By integration by parts, since $p>r$ it follows that
                \begin{align*}
                    \mathbb{E}\left|\sum_{k=1}^{K}v_k X_k\right|^r&=r\int_{0}^{\infty} t^{r-1}\mathbb{P}\left\{\left|\sum_{k=1}^{K}v_kX_k\right|>t\right\}\,dt\\
                    &=r\int_{0}^{1} t^{r-1}\mathbb{P}\left\{\left|\sum_{k=1}^{K}v_kX_k\right|>t\right\}\,dt
                    +r\int_{1}^{\infty} t^{r-1}\mathbb{P}\left\{\left|\sum_{k=1}^{K}v_kX_k\right|>t\right\}\,dt\\
                    &\leq r\int_{0}^{1}t^{r-1}\,dt + r\int_{1}^{\infty} t^{r-1}c_p\sigma^p \|v\|_p^p t^{-p}\,dt\\
                    &=1+rc_p\sigma^p\|v\|_p^p\int_{1}^{\infty}\frac{1}{t^{1+(p-r)}}\,dt\\
                    &=1+rc_p\sigma^p\|v\|_p^p\frac{1}{-(p-r)}\frac{1}{t^{p-r}}\bigg\rvert_{1}^{\infty}\\
                    &=1+\frac{1}{p-r}rc_p\sigma^p\|v\|_p^p<\infty.
                \end{align*}
                We used in the inequality that $\sum_{k=1}^{K}v_kX_k$ is also a $p-$stable random variable.\\
                Now, by the $p-$stability property (\ref{p-stability}), we also have
                \begin{equation}\label{p-property}
                    \|\sum_{k=1}^{K}v_kX_k\|_{r}=c_{r,p}\left(\sum_{k=1}^{K}|v_k|^p\right)^{1/p}.
                \end{equation}
                which shows convergence in $L_r$. In this way, the map $v\mapsto\sum_{k=1}^{K}v_kX_k$ is a bounded map.\\
                Linearity of the map follows from standard arguments.
            \item[(ii)] Taking $r=1$ in the $p-$stability property (\ref{p-property}) we obtain that
                \[
                    \|v\|_p=\left(\sum_{k=1}^{K}|v_k|^p\right)^{1/p}\leq C_{1,p}\mathbb{E}\left|\sum_{k=1}^{K}v_kX_k\right|,
                \]
                which is the claim when $k\to \infty$.
        \end{enumerate}
    \end{proof}

    The next Lemma and Theorem are the analogs for $p-$stable random variables of the results of Maurer in the case $p=2$.

    \begin{lemma}
        Let $X$ be a nontrivial $p-$stable random variable for $1<p<2$. Then there exists a constant $C(p)<\infty$ such that for any set $S$ and functions $\psi:S\to\mathbb{R}, \phi:S\to\ell_p$ and $f:S\to\mathbb{R}$ satisfying:
        \[
            \forall s,s'\in S, \psi(s)-\psi(s')\leq \|\phi(s)-\phi(s')\|_{p}
        \]
        we have
        \[
            \mathbb{E}\sup_{s\in S}\epsilon\psi(s)+f(s)\leq C(p)\mathbb{E}\sup_{s\in S}\sum_{k=1}^{K}X_k\phi(s)_k+f(s)
        \]
        where the $X_k$ are independent copies of $X$ for $1\leq k\leq \infty$ and $\phi(s)_k$ is the $k-th$ coordinate of $\phi(s)$.
    \end{lemma}
    \begin{proof}
        We take the constant in proposition \ref{prop} as $C(p)$  and let $Y=C(p)X$ and $Y_k=C(p)X_k$. So for every $v\in \ell_p$
        \[
            \|v\|_p\leq \mathbb{E}\left|\sum_{k}v_kY_k\right|.
        \]
        Let $\delta>0$ arbitrary. Then there exists $s_1^*$ and $s_2^*$ such that
        \begin{align*}
            2\mathbb{E}\sup_{s\in S}(\epsilon \psi(s)+f(s))-\delta&=\sup_{s_1,s_2\in S}\psi(s_1)+f(s_1)-\psi(s_2)+f(s_2)-\delta\\
            &\leq \psi(s_1^*)-\psi(s_2^*)+f(s_1^*)+f(s_2^*)\\
            &\leq \|\phi(s_1^*)-\phi(s_2^*)\|_p+f(s_1^*)+f(s_2^*)\\
            &\leq \mathbb{E}\left|\sum_kY_k(\phi(s_1^*)_k-\phi(s_2^*)_k)\right|+f(s_1^*)+f(s_2^*)\\
            &\leq \mathbb{E}\sup_{s_1,s_2\in S}\left|\sum_k Y_k \phi(s_1)_k-\sum_k Y_k\phi(s_2)_k\right|+f(s_1)+f(s_2)\\
        \end{align*}
        Notice we can drop the absolute value because for any fixed configuration of the $Y_k$ the maximum will be attained when the difference is positive since the remaining part $f(s_1)+f(s_2)$ is invariant under the exchange of $s_1$ and $s_2$. In this way:
        \begin{align*}
            &\phantom{=}\mathbb{E}\sup_{s_1,s_2\in S}\left|\sum_k Y_k \phi(s_1)_k-\sum_k Y_k\phi(s_2)_k\right|+f(s_1)+f(s_2)\\
            &=\mathbb{E} \sup_{s_1\in S} \sum_k Y_k \phi(s_1)_k+f(s_1)+\mathbb{E}\sup_{s_2\in S}-\sum_k Y_k\phi(s_2)_k+f(s_2)\\
            &=2\left(\mathbb{E}\sup_{s\in S}\sum_k Y_k\phi(s)_k+f(s)\right).
        \end{align*}
        The last equality follows from the symmetry of the variables $Y_k$.
    \end{proof}

    The following is the main Theorem and corresponds to an extension of Maurer's work.

    \begin{theorem}[Vector contraction-inequality for $p-$stable variables]\label{p-theorem}
        Let $X$ be a nontrivial $p-$stable random variable for $1<p<2$. Then there exists a constant $C(p)<\infty$, depending only on $p$ and the distribution of $X$, such that for any set $S$ and functions $\psi_i:S\to \mathbb{R}, \phi_i:S\to\ell_p, (1\leq i\leq n)$ satisfying
        \[
            \forall s,s'\in S, \psi_i(s)-\psi_i(s')\leq \|\phi_i(s)-\phi_i(s')\|_p
        \]
        we have
        \[
            \mathbb{E}\sup_{s\in S}\sum_i\epsilon_i\psi_i(s)\leq C(p)\mathbb{E}\sup_{s\in S}\sum_{i,k}X_{ik}\phi_i(s)_k
        \]
        where the $X_{ik}$ are independent copies of $X$ for $1\leq i\leq n$ and $1\leq k\leq \infty$ and $\phi_i(s)_k$ is the $k-th$ coordinate of $\phi_i(s)$.
    \end{theorem}
    \begin{proof}
        The constant $C(p)$ and the $Y_k$ are chosen as in the previous Lemma, recall $Y_k=C(p)X_k$. By induction on $n$, we shall prove for all $m\in\{0,...,n\}$ we have
        \[
            \mathbb{E}\sup_{s\in S}\sum_{1\leq i\leq n}\epsilon_i\psi_{i}(s)\leq \mathbb{E}\left[\sup_{s\in S}\sum_{i:1\leq i \leq m}\sum_{k}Y_{ik}\phi_i(s)_k
            +\sum_{i: m<i\leq n}\epsilon_i\psi_i(s)\right].
        \]
        So the result will follow from the case $m=n$. Clearly the case $m=0$ is an identity. Assume that the claim holds for $m-1$. Denote
        $\mathbb{E}_m=\mathbb{E}[.|\{\epsilon_i,Y_{ik}:i\neq m\}]$ and define $f:S\to \mathbb{R}$ by
        \[
            f(s)=\sum_{i:1\leq i<m}\sum_k Y_{ik}\phi_i(s)_k+\sum_{i:m<i\leq n}\epsilon_i\psi_i(s).
        \]
        Then
        \begin{align*}
            \mathbb{E}\sup_{s\in S}\sum_{1\leq i\leq n}\sigma_i \psi_i(s)&\leq \mathbb{E}\left[\sup_{s\in S}\sum_{i:1\leq i<m}\sum_k Y_{ik}\phi_i(s)_k
            +\sum_{i:m\leq i\leq n}\epsilon_i\psi_i(s)\right]\\
            &=\mathbb{E}\hspace{0.15cm}\mathbb{E}_m\sup_{s\in S}(\epsilon_m\psi_m(s)+f(s))\\
            &\leq \mathbb{E}\hspace{0.15cm}\mathbb{E}_m \sup_{s\in S}\sum_k Y_{mk}\phi_m(s)_k+f(s)\\
            &=\mathbb{E} \sup_{s\in S} \sum_{i:1\leq i\leq m}\sum_k Y_{ik}\phi_i(s)_k+\sum_{i:m<i\leq n}\epsilon_i\psi_i(s)\\
            &=C(p)\mathbb{E} \sup_{s\in S} \sum_{i:1\leq i\leq m}\sum_k X_{ik}\phi_i(s)_k+\sum_{i:m<i\leq n}\epsilon_i\psi_i(s),
        \end{align*}
        where the first inequality is true by induction hypothesis and the second is due to the previous Lemma.
    \end{proof}

    It is attractive to conjecture the following inequality for $p-$norms:
    \begin{conjecture*}
        Let $\mathcal{X}$ be any set, $n\in\mathbb{N}, (x_1,...,x_n)\in \mathcal{X}^n$, let $\mathcal{F}$ be a class of functions $f:\mathcal{X}\to\ell_p, 0<p\leq 2$ and
        let $h:\ell_p\to\mathbb{R}$ have Lipschitz norm $L$. Then
        \[
            \mathbb{E}\sup_{f\in\mathcal{F}}\sum_i\epsilon_i h(f(x_i))\leq KL\mathbb{E}\sup_{f\in\mathcal{F}}\|\sum_i \epsilon_i f(x_i)\|_p,
        \]
        where $K$ is some universal constant.
    \end{conjecture*}

    Maurer stated this conjecture for $p=2$, and he showed it is false by giving a counter-example. A similar example also disproves this conjecture for $1<p<2$:
    Let $\mathcal{X}=\ell_p$ with the canonical basis $(e_i)$ and set $x_i=e_i$ for $1\leq i \leq n$. Let $F$ be the unit ball in the set of bounded operators $B(\ell_p), 0<p\leq 2$ and for $h$ consider $h:x\in\ell_p \to \|x\|_p$ which has Lipschitz norm equal to one. If the conjecture were true then there would be a universal constant $K$ such that
    \[
        \mathbb{E}\sup_{T\in B(\mathcal{H}):\|T\|_{\infty}\leq 1}\sum_i\epsilon_i\|Tx_i\| \leq K \mathbb{E} \sup_{T\in B(\mathcal{H}):\|T\|_{\infty}\leq 1}
        \|\sum_i \epsilon_i Tx_i\|_p.
    \]
    For any Rademacher sequence $\epsilon=(\epsilon_i)$, let $T_{\epsilon}$ be the operator defined by $T_{\epsilon}e_i=e_i$ if $i\leq n$ and $\epsilon_i=1$ and $T_{\epsilon}=0$ otherwise. It follows $\|T_{\epsilon}\|_{\infty}\leq 1$. Then
    \[
        \frac{n}{2}=\mathbb{E}|\{ i: \epsilon_i=1\}|=\mathbb{E}\sum_i \epsilon_i\|T_{\epsilon}x_i\| \leq \mathbb{E}\sup_{T\in B(\mathcal{H}):\|T\|_{\infty}\leq 1} \sum_i\epsilon_i\|Tx_i\|.
    \]
    On the other hand
    \[
        \mathbb{E}\sup_{T\in B(\mathcal{H}):\|T\|_{\infty}\leq 1}\|\sum_i \epsilon_i Tx_i\|_p\leq \mathbb{E}\|\sum_i \epsilon_i e_i\|_p \leq \sqrt[p]{n}.
    \]
    Thus if the conjecture were true we would obtain $n/2\leq K\sqrt[p]{n}$ for some universal constant $K$. Note this is absurd when $1<p\leq 2$ but otherwise the inequality makes sense.

    It is interesting to point out the following: The range of exponents for which the conjecture makes sense coincides with the ranges of exponents of the vector contraction inequality for $p-$stable random variables (Theorem \ref{p-theorem}). The fact that Theorem \ref{p-theorem} can't be extended beyond $p\leq1$ is because of a technicality in Proposition \ref{prop}, so we know $p=1$ is an exponent to look at it. It would be interesting to know if the conjecture is true or not for the case $p=1$.
    We have explored the following case:

    For $p=1$, let $\mathcal{X}$ be any set, $n\in\mathbb{N}, (x_1,...,x_n)\in \mathcal{X}^n$, let $\mathcal{F}$ be a collection of functions whose image is the unit sphere of $\ell_1$, that is, $\mathcal{F}=\{f|f:\mathcal{X}\to S_{\ell_1}\}$ and let $h:\ell_1 \to \mathbb{R}$ defined by $x\mapsto \|x\|_1$.
    On the left hand side of the conjecture we have
    \begin{align*}
        \mathbb{E}\sup_{f\in\mathcal{F}}\sum_i\epsilon_i h(f(x_i)) &= \mathbb{E}\sup_{f\in\mathcal{F}}\sum_i\epsilon_i \|f(x_i)\|_1\\
         &=\mathbb{E}\sum_i\epsilon_i =0.
    \end{align*}
    On the right side we obtain
    \begin{align*}
        \mathbb{E}\sup_{f\in\mathcal{F}}\|\sum_i \epsilon_i f(x_i)\|_1 \leq \mathbb{E}\sup_{f\in\mathcal{F}}\|\sum_i \epsilon_i e_i\|_1 = n.
    \end{align*}
    Thus, the conjecture is true for this particular family of functions. Whether the conjecture is true or not for the family of functions whose image is bounded might be a case worth trying.

\end{document}